\documentclass[11pt]{amsart}
\usepackage[top = 1in, bottom = 1in, left = 1.5in, right=1.5in,
marginparwidth=1.3in]{geometry}
\usepackage{setspace}
\usepackage{amsmath, amsfonts, amsthm, amstext, xspace}
\usepackage[mathscr]{euscript}
\usepackage{amssymb,latexsym}
\usepackage{comment}
\usepackage{xcolor}
\usepackage{hyperref}
\usepackage[capitalise]{cleveref}
\definecolor{seagreen}{RGB}{46,139,87}
\definecolor{maroon}{RGB}{128,0,0}
\definecolor{darkviolet}{RGB}{148,0,211}
\definecolor{twelve}{RGB}{100,100,170}
\definecolor{thirteen}{RGB}{100,150,50}
\definecolor{fourteen}{RGB}{200,0,0}
\definecolor{fifteen}{RGB}{0,200,0}
\definecolor{sixteen}{RGB}{0,0,200}
\definecolor{seventeen}{RGB}{200,0,200}
\definecolor{eighteen}{RGB}{0,200,200}

\usepackage{centernot}
\usepackage{longtable}
\usepackage{mathtools}
\usepackage{stmaryrd}
\makeatletter
\newcommand{\xMapsto}[2][]{\ext@arrow 0599{\Mapstofill@}{#1}{#2}}
\def\Mapstofill@{\arrowfill@{\Mapstochar\Relbar}\Relbar\Rightarrow}
\makeatother

\usepackage[all]{xy}
\newtheorem{thm}{Theorem}[section]
\newtheorem*{theorem*}{Theorem}
\newtheorem*{conjecture*}{Conjecture}
\newtheorem*{corollary*}{Corollary}
\newtheorem{lemma}[thm]{Lemma}

\newtheorem{corollary}[thm]{Corollary}

\newtheorem{proposition}[thm]{Proposition}

\theoremstyle{definition}

\newtheorem{example}[thm]{Example}
\newtheorem{remark}[thm]{Remark}

\newtheorem{question}[thm]{Question}

\numberwithin{equation}{section} 

\def\r{\mathbb{R}}

\def\t{\mathbb{T}}
\def\z{\mathbb{Z}}

\def\cp{\mathcal{P}}

\def\coker{\operatorname{coker}}

\def\num{\operatorname{num}}

\definecolor{llteal}{RGB}{198,232,227}
\definecolor{llred}{RGB}{237,228,228}
\definecolor{llgray}{RGB}{230,230,230}
\definecolor{maroon}{RGB}{150,0,0}
\definecolor{orange}{RGB}{255,165,0}

\newcommand{\highlight}[1]{\ifmmode{\text{\sethlcolor{llgray}\hl{$#1$}}}\else{\sethlcolor{llred}\hl{#1}}\fi}

\usepackage{amssymb}
\usepackage{tikz-cd}

\author{J.D. Quigley}\address{Max Planck Institute for Mathematics}\email{jquigley1993@gmail.com}
\title{Some smooth circle and cyclic group actions on exotic spheres}

\begin{document}
\maketitle

\begin{abstract}
Classical work of Lee, Schultz, and Stolz relates the smooth transformation groups of exotic spheres to the stable homotopy groups of spheres. In this note, we apply recent progress on the latter to produce smooth circle and cyclic group actions on certain exotic spheres. 
\end{abstract}

\tableofcontents

\section{Introduction}

One appealing feature of spheres is their high degree of symmetry. For example, for all $n \geq 1$, the $n$-sphere $S^n$ equipped with its standard smooth structure has smooth rotational symmetry, i.e., it supports a nontrivial smooth action of the circle group $\t$. More generally, by regarding $S^n$ as the unit sphere in $\r^{n+1}$, we see that $S^n$ admits a smooth $SO(n+1)$-action. In a precise sense, cf. \cite[Sec. 0]{Str94}, this implies that spheres are the ``most symmetric" of all smooth manifolds. 

However, if we consider \emph{exotic} $n$-spheres, i.e., spheres which are homeomorphic but not diffeomorphic to  $S^n$, then this high degree of symmetry is usually not present, cf. \cite{Hsi67, HH67, HH69, LY74, Str94}. In 1985, Schultz \cite{Sch85} highlighted the following questions concerning the rotational symmetry of exotic spheres:

\begin{question}[{\cite{Sch85}}]\label{Question:Schultz}
Let $\Sigma^n$ be an exotic $n$-sphere, $n \geq 5$. Does $\Sigma^n$ support a nontrivial smooth $\t$-action? Does $\Sigma^n$ support a nontrivial smooth $\z/p$-action for every prime $p$? 
\end{question}

Bredon \cite{Bre67}, Schultz \cite{Sch75}, and Joseph \cite{Jos81} have produced examples of nontrivial smooth $\t$-actions on certain $(8k+1)$- and $(8k+2)$-dimensional exotic spheres. Schultz has also shown that every $8$- and $10$-dimensional exotic sphere admits a smooth semi-free $\t$-action with fixed point set $S^4$ \cite{Sch72} and that for each odd prime $p$, a certain $(2p^2-2p-2)$-dimensional exotic sphere admits a nontrivial smooth $\t$-action \cite{Sch73}; related examples are discussed in \cite{Sch85}, cf. \cref{Cor:Schultz}. Nonexistence results, as well as the existence of some nontrivial smooth $\z/p$-actions, are discussed in \cite{Sch78}. 

The purpose of this note is to provide some positive answers to \cref{Question:Schultz}. Our approach is as follows. Classical results of Lee and Schultz (\cref{Thm:Manyp}) and Schultz and Stolz (\cref{Thm:SS}) link \cref{Question:Schultz} to the stable homotopy groups of spheres, the Kervaire--Milnor sequences \cite{KM63}, and the Mahowald invariant \cite{MR93}. Using recent progress on the stable homotopy groups of spheres and Mahowald invariants, we are able to prove the existence of nontrivial smooth $\t$- and $\z/p$-actions on many exotic spheres up to dimension $100$. We also describe arithmetic conditions on the pair $(n,p)$ under which every exotic $n$-sphere admits a smooth free $\z/p$-action. 

The first result we will use was proven by C.N. Lee in \cite{Lee68} and refined by Schultz in \cite{Sch85}:

\begin{thm}[{\cite{Lee68}, \cite[Thm. 1.9]{Sch85}}]\label{Thm:Manyp}
Let $\Sigma^n$ be an exotic sphere. For all but finitely many primes $p$, $\Sigma^n$ supports a nontrivial smooth $\z/p$-action. More precisely, $\Sigma^n$ admits a free smooth $\z/p$-action whenever $p$ is prime to the order of $\Sigma^n$ in $\Theta_n$. 
\end{thm}

Here, $\Theta_n$ is the group of h-cobordism classes of homotopy $n$-spheres introduced by Kervaire and Milnor in \cite{KM63}. The group $\Theta_n$ sits in exact sequences with $\Theta_n^{bp}$, the subgroup of h-cobordism classes of homotopy $n$-spheres which bound stably parallelizable manifolds, and $\coker(J_n)$, the cokernel of the J-homomorphism. 

In \cref{Sec:KM}, we apply results on the order of $\Theta^{bp}_n$ from \cite{KM63, Lev85} and $\coker J_n$ from \cite{BHHM20, BMQ22, IWX20, Rav86} to study the order of $\Theta_n$. In \cref{Thm:Little} and \cref{Cor:Funny}, we provide arithmetic conditions on the dimension $n$ and odd prime $p$ under which every exotic $n$-sphere admits a smooth free $\z/p$-action. In \cref{App:Tables}, we compute the prime factors of $\Theta_n$ for all $n \leq 100$; combined with \cref{Thm:Manyp}, this allows us to deduce the existence of nontrivial smooth $\z/p$-actions on many exotic spheres:

\begin{thm}
The prime factors of $\Theta^{bp}_n$ and $\coker J_n$ are listed for $n \leq 100$ in \cref{App:Tables}. Every exotic $n$-sphere admits a smooth free $\z/p$-action for each prime $p$ which does not appear in row $n$. 
\end{thm}

\begin{example}\label{Exm:LS23}
The prime factors of $\coker(J_{23})$ are $2$ and $3$, and the prime factors of $\Theta_{23}^{bp}$ are $2$, $23$, $89$, and $691$. Therefore each exotic $23$-sphere admits a smooth free $\z/p$-action for all primes $p \notin \{2, 3, 23, 89, 691\}$. 
\end{example}

\Cref{Thm:Manyp} leaves open the question of nontrivial smooth $\z/p$-actions on exotic $n$-spheres whose order in $\Theta_n$ is divisible by $p$. Using beautiful results of Schultz and Stolz, it is sometimes possible to prove that such exotic spheres admit nontrivial smooth $\t$- or $\z/p$-actions. To state their results, let $\cp(-)$ denote the Pontryagin--Thom construction identifying framed bordism classes of stably framed smooth $n$-manifolds with classes in the $n$-th stable homotopy group of spheres $\pi_n^s$ (cf. \cite[Sec. IX.5]{Kos93}), and let $M(-)$ be the Mahowald invariant \cite{MR93} which associates a nontrivial coset in the stable homotopy groups of spheres to any nontrivial element in the stable homotopy groups of spheres. 

\begin{thm}[{\cite[Thm. 3.7]{Sch85} for $p$ odd, \cite[Thm. D]{Sto88} for $p=2$}]\label{Thm:SS} \

\begin{enumerate}
\item Let $p$ be an odd prime. Let $0 \neq \alpha \in \pi_n^s$ and $\beta = M(\alpha) \in \coker J_m$. Suppose $\Sigma_0$ is a framed sphere with $\cp(\Sigma_0) = \alpha$. 
\begin{enumerate}
\item If $m-n$ is even, then there is an exotic sphere $\Sigma_1$ such that $\cp(\Sigma_1) = \beta$ and $\Sigma_1$ admits a smooth $\z/p$-action with fixed point set $\Sigma_0$. 
\item If $m-n$ is odd, then there is an exotic sphere $\Sigma_1$ such that $\cp(\Sigma_1)=\beta$ and $\Sigma_1$ admits a smooth $\t$-action with an $(n-1)$-dimensional fixed point set. 
\end{enumerate}

\item Let $p=2$. Let $0 \neq \alpha \in \pi_n^s$ and $\beta = M(\alpha) \in \coker J_m$. Suppose $\Sigma_0$ is a framed sphere with $\cp(\Sigma_0) = \alpha$ and let $\Sigma_1$ be an exotic sphere with $\cp(\Sigma_1) = \beta$. If $m > 2n+1$ and either $m$ and $m-n$ are both odd or $m-n$ is even and $m \equiv 1 \mod 4$, then there exists a smooth $\z/2$-action on the connected sum $\Sigma_1 \# \Sigma'$, $\Sigma' \in \Theta^{bp}_{m+1}$, with fixed point set $\Sigma_0$. 
\end{enumerate}
\end{thm}

\begin{remark}
We believe there is a slight ambiguity in \cite[Thm. 3.7]{Sch85}. On page 248 of \emph{loc. cit.}, Schultz takes the Pontryagin--Thom construction to be a map from $\Theta_n/\Theta_n^{bp}$ to $\coker J_n$ and defines the Mahowald invariant (page 259, \emph{loc. cit}) only on elements in $\coker J$. However, the Pontryagin--Thom isomorphism and Mahowald invariant apply to \emph{any} element in the stable homotopy groups of spheres, not just elements in $\coker J$. 

We believe a more applicable version of \cite[Thm. 3.7]{Sch85} appears as \cite[Thm. C]{Sto88}. There, one begins with a pair of framed spheres $\Sigma_0$ and $\Sigma_1$, identifies them with elements in the stable homotopy groups of spheres via the Pontryagin--Thom construction, and assumes that $M(\Sigma_0) = \Sigma_1$. In particular, Schultz's principle example \cite[Ex. 3.8]{Sch85} works under these weaker hypotheses. 
\end{remark}

In \cref{Sec:MI}, we recall many Mahowald invariant computations which were made after the original appearance of \cref{Thm:SS} and apply them to prove the existence of nontrivial smooth $\t$- and $\z/p$-actions on certain exotic spheres whose order in $\Theta_n$ is divisible by $p$. 

\begin{thm}
Some combinations of dimensions $n$ and primes $p$ for which \cref{Thm:SS} implies the existence of a nontrivial smooth $\t$- or $\z/p$-action on an exotic $n$-sphere whose order in $\Theta_n$ is divisible by $p$ are marked with an asterisk $*$ in \cref{App:Tables}. 
\end{thm}

\begin{example}
In \cref{Exm:LS23}, we mentioned that \cref{Thm:Manyp} cannot be used to produce a nontrivial smooth $\z/3$-action on certain exotic $23$-spheres since $3$ divides the order of $\Theta_{23}$. However, Behrens \cite{Beh06} showed that 
$$M(\alpha_2) \doteq \beta_1^2 \alpha_1,$$
where $\alpha_2 \in (\pi_7^s)_{(3)}$ and $\beta_1^2 \alpha_1 \in (\pi_{23}^s)_{(3)}$. Applying the odd-primary part of \cref{Thm:SS}, we find that the exotic $23$-sphere corresponding to $\beta_1^2 \alpha_1$, whose order in $\Theta_{23}$ is divisible by $3$, supports a nontrivial smooth $\z/3$-action with fixed points the standard $7$-sphere. 
\end{example}

Some suggestions for future work and remarks on the limitations of these techniques appear in \cref{SS:Further}.

\subsection{Acknowledgments}

The author thanks Guchuan Li, Stephan Stolz, and Paul VanKoughnett for helpful discussions, as well as the Max Planck Institute for providing a wonderful working environment and financial
support. 

\section{Smooth free $\z/p$-actions via the Kervaire--Milnor sequences}\label{Sec:KM}

In order to apply \cref{Thm:Manyp}, we need to analyze the orders of the groups $\Theta_n$ of homotopy $n$-spheres. Kervaire and Milnor showed in \cite{KM63} that there is an isomorphism
$$\Theta_{4k} \cong \coker J_{4k}$$
and exact sequences
$$0 \to \Theta^{bp}_{2k+1} \to \Theta_{2k+1} \to \coker J_{2k+1} \to 0,$$
$$0 \to \Theta_{4k+2} \to \coker J_{4k+2} \xrightarrow{\Phi} \z/2 \to \Theta^{bp}_{4k+1} \to 0,$$
where $\Phi$ is the Kervaire invariant \cite{Ker60}. In this section, we will apply information about $\Theta^{bp}_n$ and $\coker J$ to analyze the order of $\Theta_n$. Our main result (\cref{Thm:Little}) provides arithmetic conditions on the natural number $n$ and odd prime $p$ under which every exotic $n$-sphere admits a smooth free $\z/p$-action. 

Our first lemma allows us to ignore contributions from $\coker J$, provided that $n$ is sufficiently small. 

\begin{lemma}
Let $p$ be an odd prime. Then 
$$(\coker J_n)_{(p)} = 0$$
for $n < 2(p^2-1)-2(p-1)-2$.
\end{lemma}

\begin{proof}
By \cite[Sec. 5.3]{Rav86}, the first nontrivial element in $(\coker J_*)_{(p)}$ is the Greek letter element $\beta_1$, which appears in stem $2(p^2-1)-2(p-1)-2$.
\end{proof}

We now turn out attention to $\Theta_n^{bp}$. The following theorem is proven in \cite{KM63, Lev85}.

\begin{thm}\label{Thm:Levine}
The subgroup $\Theta_n^{bp} \subseteq \Theta_n$ of h-cobordism classes of homotopy $n$-spheres which bound a stably parallelizable manifold is given by
\[
\Theta_n^{bp} \cong \begin{cases}
	0 \quad & \text{ if } n \equiv 0 \mod 2, \\
	0 \quad & \text{ if } n=4k+1\in \{1,5,13,29,61, \text{ and possibly }125\}, \\
	\z/2 \quad & \text{ if } n=4k+1 \notin \{1,5,13,29,61, \text{ and possibly } 125\}, \\
	\z/t_k \quad & \text{ if } n=4k-1, 
\end{cases}
\]
where 
$$t_k := \dfrac{3-(-1)^{k}}{2} 2^{2k-2}(2^{2k-1}-1) \cdot \operatorname{num}\left( \dfrac{B_{2k}}{4k} \right)$$
with $B_j$ the $j$-th Bernoulli number. 
\end{thm}

\begin{corollary}
Let $p$ be an odd prime. For $n < 2(p^2-1)-2(p-1)-2$,
\[
(\Theta_n)_{(p)} \cong 
\begin{cases}
	(\z/t_k)_{(p)} \quad & \text{ if } n=4k-1, \ k \geq 2, \\
	0 \quad & \text{ otherwise.}
\end{cases}
\]
\end{corollary}

\begin{lemma}
Let $p$ be an odd prime. Then $(\z/t_k)_{(p)} \neq 0$ if and only if $p | 2^{2k-1}-1$ or $p| \num \left( \frac{B_{2k}}{2k} \right)$. 
\end{lemma}

\begin{proof}
We have $(\z/t_k)_{(p)} \neq 0$ if and only if $p | t_k$, and since $p$ is odd, 
\begin{align*}
p \mid t_k &\iff p \mid \dfrac{3-(-1)^{k}}{2} 2^{2k-2}(2^{2k-1}-1) \operatorname{num}\left(\dfrac{B_{2k}}{4k} \right) \\
&  \iff p | (2^{2k-1}-1)\operatorname{num}\left( \dfrac{B_{2k}}{2k} \right) \\
& \iff p | 2^{2k-1}-1 \text{ or } p | \operatorname{num}\left( \dfrac{B_{2k}}{2k} \right).
\end{align*}
\end{proof}

Putting these observations together, we have shown:

\begin{proposition}\label{Thm:Little}
Let $p$ be an odd prime and let $n < 2(p^2-1)-2(p-1)-2$. If $n \equiv 0,1,2 \mod 4$, or if $n = 4k-1$ with $p \nmid 2^{2k-1}-1$ and $p \nmid \operatorname{num}(\frac{B_{2k}}{2k})$, then every exotic $n$-sphere admits a smooth free $\z/p$-action.
\end{proposition}

In the following two remarks, we give some examples of when the conditions in \cref{Thm:Little} are satisfied. 

\begin{remark}
We first consider the condition $p \nmid 2^{2k-1}-1$. If $2k-1$ is prime, then quadratic reciprocity implies that the prime factors of $2^{2k-1}-1$ must be congruent to $\pm 1$ modulo $8$. Thus if $2k-1$ is prime and $p \not\equiv \pm 1 \mod 8$, the condition $p \nmid 2^{2k-1}-1$ is satisfied.
\end{remark}

\begin{remark}
We can also say something about the condition $p \nmid \operatorname{num}(\frac{B_{2k}}{2k})$. As explained in the proof of \cite[Thm. 1]{Tha12}, the prime factors of $\operatorname{num}(\frac{B_{2k}}{2k})$ must be irregular.\footnote{Recall that a prime $q$ is irregular if $q | B_j$ for some even $j \leq p-3$.} Thus if $p$ is a regular prime, the condition $p \nmid \operatorname{num}(\frac{B_{2k}}{2k})$ is satisfied. 
\end{remark}

\begin{corollary}\label{Cor:Funny}
Let $p \not\equiv \pm 1 \mod 8$ be an odd regular prime and let $n = 4k-1$ with $2k-1$ prime. Then every exotic $n$-sphere admits a smooth free $\z/p$-action. 
\end{corollary}

\section{Nontrivial smooth $\t$- and $\z/p$-actions via the Mahowald invariant}\label{Sec:MI}

As discussed in the introduction, we can find nontrivial smooth $\t$- and $\z/p$-actions on exotic spheres whose order in $\Theta_n$ is divisible by $p$ using the \emph{Mahowald invariant}. We refer the reader to \cite{MR93} for a definition, since for our purposes, it suffices to know that for each prime $p$, the Mahowald invariant is a construction which assigns a nontrivial coset in the $p$-local stable homotopy groups of spheres to each nontrivial element in the $p$-local stable homotopy groups of spheres. We will freely use the names of elements in the stable homotopy groups of spheres from \cite{Rav86} (e.g., for elements like $\alpha_i$, $\beta_j$) and \cite{IWX20} (e.g., for elements like $\eta \eta_4$, $\kappa^2$). 

Our recollection of Mahowald invariant computations is divided into two sections. In \cref{SS:Families}, computations of infinite families of Mahowald invariants are discussed, and in \cref{SS:LowDim}, additional low-dimensional computations at small primes are recalled. After each result is stated, its consequences for transformation groups of exotic spheres via \cref{Thm:SS} are given. Directions for future work and a discussion of the limitations of these techniques appear in \cref{SS:Further}. 

\subsection{Mahowald invariants of infinite families}\label{SS:Families}

Computing the Mahowald invariants of infinite families of elements in the stable homotopy groups of spheres is a difficult problem in stable homotopy theory. This section recalls almost all of the existing computations in this direction. 




\begin{thm}[{\cite[Thm. 3.5]{MR93}, \cite[Cor. 1.4]{Sad92}}]\label{Thm:Malphai}
For all $i > 0$ and primes $p \geq 5$, 
$$M(\alpha_i) = \beta_i.$$
\end{thm}

\begin{thm}[{\cite[Thm. 15.7]{Beh06}}]\label{Thm:Malpha3}
For all $i>0$ with $i \equiv 0,1,5 \mod 9$ and $p=3$,
$$M(\alpha_i) = (-1)^{i+1}\beta_i.$$
\end{thm}

\begin{corollary}[{Compare with \cite[Ex. 3.8]{Sch85} for $p \geq 5$}]\label{Cor:Schultz}
For all $i \geq 1$ and primes $p \geq 5$, and for all $i \geq 1$ with $i \equiv 0,1,5 \mod 9$ and $p=3$, the exotic sphere $\Sigma^{2(p^2-1)i-2(p-1)-2}$ corresponding to $\beta_i \neq 0 \in \coker J_{2(p^2-1)i-2(p-1)-2}$ supports a smooth $\t$-action with a $(2(p-1)i-2)$-dimensional fixed point set. 
\end{corollary}

\begin{remark}
The case $i=1$ was proven by more geometric methods in \cite{Sch73}. 
\end{remark}

\begin{thm}[{\cite[Sec. 6]{Sad92}}]\label{Thm:Sad}
For all primes $p \geq 5$, 
$$\beta_{p/2} \in M(\alpha_{p/2}).$$
\end{thm}

\begin{corollary}
For all primes $p \geq 5$, the exotic sphere $\Sigma^{2(p^2-1)p - 4(p-1)-2}$ corresponding to $\beta_{p/2} \neq 0 \in \coker J$ supports a smooth $\t$-action with a $(2(p-1)p-2)$-dimensional fixed point set. 
\end{corollary}

\subsection{Low-dimensional Mahowald invariants at small primes}\label{SS:LowDim}

The difficulty of computing $p$-primary Mahowald invariants increases when $p$ is a small prime. In \cite{Beh06, Beh07}, Behrens introduced new techniques which allowed for the computation of new $3$- and $2$-primary Mahowald invariants in low dimensions.

\begin{proposition}[{\cite[Prop. 12.1]{Beh06}}]\label{Thm:Low3}
The following Mahowald invariants hold at $p=3$:
$$M(\alpha_2) \doteq \beta_1^2 \alpha_1; \quad M(\alpha_{3/2}) = -\beta_{3/2}; \quad M(\alpha_3) = \beta_3;$$
$$M(\alpha_4) \doteq \beta_1^5; \quad M(\alpha_{6/2}) = \beta_{6/2}; \quad M(\alpha_6) = -\beta_6.$$
Here, we write `$\doteq$' for equations which hold up to multiplication by a unit. 
\end{proposition}

\begin{corollary}
Let $p=3$. 
\begin{enumerate}
\item The exotic sphere $\Sigma^{23}$ corresponding to $0 \neq \beta_1^2 \alpha_1 \in \coker J_{23}$ supports a smooth $\z/3$-action wtih fixed points $S^7$. 
\item The exotic spheres $\Sigma^{38}$, $\Sigma^{42}$, $\Sigma^{50}$, $\Sigma^{86}$, and $\Sigma^{90}$ corresponding to the nontrivial elements $\beta_{3/2}$, $\beta_3$, $\beta_1^5$, $\beta_{6/2}$, and $\beta_6$ in $\coker J$, respectively, support smooth $\t$-actions with $10$-, $10$-, $14$-, $22$-, and $22$-dimensional fixed point sets, respectively. 
\end{enumerate}
\end{corollary}

\begin{thm}[{Part of \cite[Thm. 11.1]{Beh07}}]\label{Thm:Low2}
The following Mahowald invariants hold at $p=2$:
$$M(\eta^2) = \nu^2; \quad M(\eta^3) = \nu^3; \quad M(2\nu) = \sigma \eta; \quad M(\sigma) = \sigma^2; \quad M(2\sigma) = \eta_4; $$
$$M(4\sigma) = \eta \eta_4; \quad M(8\sigma) = \eta^2 \eta_4; \quad M(\eta \sigma) = \nu^*; \quad M(\eta^2 \sigma) = \nu \nu^*; \quad M(v_1^4 \eta) = \nu \bar{\kappa};$$
$$M(v_1^4 \eta^2) = \kappa^2; \quad M(v_1^4 \eta^3) = \eta q; \quad M(v_1^4 \nu) = \nu^2 \bar{\kappa}; \quad M(v_1^4 2\nu) = q.$$
\end{thm}

\begin{corollary}
Let $p=2$. Then, potentially after taking the connected sum with elements in $\Theta_*^{bp}$, the exotic spheres $\Sigma^9$, $\Sigma^{17}$, $\Sigma^{21}$, and $\Sigma^{33}$ corresponding to the nontrivial elements $\nu^3$, $\eta \eta_4$, $\nu \nu^*$, and $\eta q$ in $coker J$, respectively, support smooth involutions with fixed points $S^3$, $S^7$, $M^9$, and $S^{11}$, where $M^9$ is the exotic $9$-sphere corresponding to $\eta^2 \sigma$. 
\end{corollary}

\subsection{Further remarks}\label{SS:Further}

We close by mentioning some directions for follow-up work and limitations of these techniques. 

\begin{remark}
Schultz mentions \cite[Pg. 260]{Sch85} the possibility of applying \cref{Thm:SS} to Mahowald invariants at higher chromatic heights. Mahowald and Ravenel state \cite{MR93} that $M(\beta_1) = \beta_1^p$, and outline an approach to showing $M(\beta_i) = \gamma_i$ for $i \geq 2$ and $p \geq 7$.\footnote{Here, the restriction $p \geq 7$ ensures that $\gamma_i$ is defined via the work of Miller--Ravenel--Wilson \cite{MRW77}.} Assuming this is true, one obtains nontrivial smooth $S^1$-actions on some additional exotic spheres. 
\end{remark}

\begin{remark}
Many exotic spheres are detected in $\coker(J)$ by divided Greek letter elements $\beta_{kp/i}$ (see, for instance, \cite[Sec. 3]{Beh07}). For example, the elements $\beta_{6/3} \in \coker(J_{82})_{(3)}$ and $\beta_{6/2} \in \coker(J_{86})_{(3)}$ detect exotic $82$- and $86$-spheres, respectively, for which a nontrivial $\z/3$-action is not guaranteed by \cref{Thm:Manyp}. As mentioned in \cref{Thm:Sad}, Sadofsky showed in \cite{Sad92} that $M(\alpha_{p/2}) = \beta_{p/2}$ for all primes $p \geq 5$. It seems plausible that $M(\alpha_{kp/i}) = \beta_{kp/i}$ for larger $k$ and $i$; if this were true, one could produce additional infinite families of nontrivial $\z/p$-actions on exotic spheres whose order in $\Theta_n$ is divisible by $p$. 
\end{remark}

\begin{remark}
Belmont and Isaksen \cite{BI22} have recently introduced some promising techniques for computing $2$-primary Mahowald invariants. It would be interesting to see how far these ideas can be pushed and their consequences for nontrivial smooth involutions on exotic spheres. 
\end{remark}


\begin{remark}[Limitations]
Fix a prime $p$. Let $G_k := (\pi_k^s)_{(p)}$, and let $R_k \subseteq G_k$ denote the subgroup generated by classes which are Mahowald invariants. In \cite[Conj. 1.13]{MR93}, Mahowald and Ravenel conjecture that
$$\lim_{k \to \infty} \dfrac{\log_p |R_k|}{\log_p |G_k|} = \dfrac{1}{p^2},$$
assuming that $\log_p |G_k|$ grows linearly in $k$. Burklund \cite{BHS22, Bur22} has recently shown that $\log_p |G_k|$ grows sublinearly, so one might expect that the limit above approaches $1/p$ instead of $1/p^2$. If this is true, then \cref{Thm:SS} can be applied to roughly $1$ out of every $p$ exotic spheres from $\coker J$ if $p$ is odd (and some smaller proportion if $p=2$). 
\end{remark}


\appendix

\section{Tables of nontrivial $\t$- and $\z/p$-actions in low dimensions}\label{App:Tables}

In this appendix, we compute the prime factors of $\coker(J_n)$ and $\Theta_n^{bp}$ in all dimensions $n \leq 100$ where exotic spheres are known to exist (see \cite{BHHM20} for a list up to dimension $140$). 

The prime factors of $\coker(J_n)$ follow directly from inspection of the $p$-local stable homotopy groups of spheres. Note that $\coker(J_n)_{(p)} = 0$ for $n \leq 100$ if $p \geq 11$, so we only need to examine $p \in \{2,3,5,7\}$:
\begin{itemize}
\item For $p=2$, this follows from recent work of Isaksen--Wang--Xu \cite{IWX20}, $n \leq 95$, and from \cite{BHHM20, BMQ22} for $96 \leq n \leq 100$.  
\item For $p \in \{3,5\}$, we use Ravenel's extensive Adams--Novikov spectral sequence computations \cite[Thms. 7.5.3, 7.6.5]{Rav86}.
\item For $p=7$, the only element in $\coker(J)_n$ with $n\leq 100$ is $\beta_1 \in \coker(J_{82})$. 
\end{itemize}

The prime factors of $\Theta_n^{bp}$ follow from \cref{Thm:Levine}. Since $\Theta_n^{bp}=0$ if $n$ is even, prime factors only appear when $n$ is odd. We can further reduce to the study of odd prime factors, since $2$ divides $|\Theta_n^{bp}|$ for all odd $n \geq 7$, except in the exceptional cases $n \in \{1,5,13,29,61\}$. The odd prime factors of $t_k$ were determined using Mathematica. 

We add an asterisk whenever some exotic sphere whose order is divisible by $p$ admits a $\t$- or $\z/p$-action via \cref{Thm:SS}. For example, for $n=9$, the prime $2$ divides the order of $\coker(J_9)$, but the element $\nu^3 \in \coker(J_9)$ is a Mahowald invariant (\cref{Thm:Low2}) for which \cref{Thm:SS} applies. Thus we have `$2^*$' in the row $n=9$ and second column, instead of just `$2$'. 

Finally, we added two asterisks to the `$2$' in the $n=30$ row, since the only nontrivial element in $\coker(J_{30})_{(2)}$ is $\theta_4$, which has Kervaire invariant one and thus does not detect an exotic sphere.

\begin{figure}[!htb]
\centering
\begin{tabular}{ | c | c | c | }
\hline
n & prime factors of & prime factors of $\Theta_n^{bp}$ \\
    &  $\coker(J_n)$   & \\
\hline
7 & 2 & 2 \\
\hline
8 & 2 & \\
\hline
9 & 2* & 2 \\
\hline
10 & 2, 3* & \\
\hline
11 & 2 & 2, 31 \\
\hline
13 & 3 & \\
\hline
14 & 2 & \\
\hline
15 & 2 & 2, 127 \\
\hline
16 & 2 & \\
\hline
17 & 2* & 2 \\
\hline
18 & 2 &  \\
\hline
19 & 2 & 2, 7, 73 \\
\hline
20 & 2, 3 &  \\
\hline
21 & 2* & 2 \\
\hline
22 & 2 & \\
\hline
23 & 2, 3* & 2, 23, 89, 691 \\
\hline
24 & 2 & \\
\hline
25 & 2 & 2 \\
\hline
26 & 2, 3 & \\
\hline 
27 & 2 & 2, 8191 \\
\hline
28 & 2 & \\
\hline
29 & 3 & \\
\hline
30 & 2**, 3 &  \\
\hline
31 & 2 & 2, 7, 31, 151, 3617 \\
\hline
32 & 2 &  \\
\hline
33 & 2* & 2 \\
\hline
34 & 2 & \\
\hline
35 & 2 & 2, 43867 \\
\hline
36 & 2, 3 &  \\
\hline
37 & 2, 3 & 2 \\
\hline
38 & 2, 3*, 5* &  \\
\hline
39 & 2, 3 & 2, 283, 617 \\
\hline
40 & 2, 3 & \\
\hline
41 & 2 & 2 \\
\hline
42 & 2, 3* & \\
\hline
43 & 2 & 2, 7, 127, 131, 337, 593 \\
\hline
44 & 2 & \\
\hline
45 & 2, 3, 5 & 2 \\
\hline
46 & 2, 3 & \\
\hline
47 & 2, 3 & 2, 47, 103, 178481, 2294797 \\
\hline
48 & 2 & \\
\hline
49 & 2, 3 & 2\\
\hline
50 & 2, 3* & \\
\hline
\end{tabular}
\caption{The prime factors of $\Theta_n^{bp}$, $7 \leq n \leq 50$. The trivial group $\Theta_{12}$ is omitted.}
\end{figure}

\begin{figure}[!htb]
\centering
\begin{tabular}{ | c | c | c | }
\hline
n & prime factors of & prime factors of $\Theta_n^{bp}$ \\
    &  $\coker(J_n)$   & \\
\hline
51 & 2 & 2, 657931 \\ 
\hline
52 & 2, 3 & \\ 
\hline
53 & 2 & 2 \\ 
\hline
54 & 2 & \\ 
\hline
55 & 2, 3 & 2, 7, 73, 9349, 262657, 362903 \\ 
\hline
57 & 2 & 2 \\ 
\hline
58 & 2 & \\ 
\hline
59 & 2 & 2, 233, 1103, 1721, 2089, 1001259881  \\ 
\hline
60 & 2 & \\  
\hline
62 & 2, 3 & \\ 
\hline
63 & 2 & 2, 37, 683, 305065927, 2147493647 \\ 
\hline
64 & 2 & \\ 
\hline
65 & 2, 3 & 2 \\ 
\hline
66 & 2 & \\ 
\hline
67 & 2 & 2, 7, 23, 89, 5999479, $\num(B_{34}/34)$ \\ 
\hline
68 & 2, 3 & \\ 
\hline
69 & 2 & 2 \\ 
\hline
70 & 2 & \\ 
\hline
71 & 2 & 2, 31, 71, 127, 122921, $\num(B_{36}/36)$ \\ 
\hline
72 & 2, 3 & \\ 
\hline
73 & 2 & 2 \\ 
\hline
74 & 2, 3* & \\ 
\hline
75 & 2, 3 & 2, 223, 616318177, $\num(B_{38}/38)$ \\ 
\hline
76 & 2, 5& \\ 
\hline
77 & 2 & 2 \\ 
\hline
78 & 2, 3 & \\ 
\hline
79 & 2 & 2, 7, 79, 8191, 121369, 137616929, 1897170067619 \\ 
\hline
80 & 2 & \\ 
\hline
81 & 2, 3 & 2 \\ 
\hline
82 & 2, 3, 7* & \\ 
\hline
83 & 2, 5 & 2, 13367, 164511353, $\num(B_{42}/42)$ \\ 
\hline
84 & 2, 3 & \\ 
\hline
85 & 2, 3 & 2 \\ 
\hline
86 & 2, 3*, 5* & \\ 
\hline
87 & 2 & 2, 59, 431, 8089, 9719, 2099863, 2947939, 1798482437 \\ 
\hline
88 & 2 & \\ 
\hline
89 & 2 & 2 \\ 
\hline
90 & 2, 3* & \\ 
\hline
91 & 2, 3 & 2, 7, 31, 73, 151, 631, 23311 \\ 
    &        & 383799511, 67568238839737 \\
\hline
92 & 2, 3 & \\ 
\hline
93 & 2, 3, 5 & 2 \\ 
\hline
94 & 2, 3 & \\ 
\hline
\end{tabular}
\caption{The prime factors of $\Theta_n^{bp}$, $51 \leq n \leq 94$. The trivial groups $\Theta_{56}$ and $\Theta_{61}$ are omitted.}
\end{figure}

\clearpage

\begin{figure}[!htb]
\centering
\begin{tabular}{ | c | c | c | }
\hline
n & prime factors of & prime factors of $\Theta_n^{bp}$ \\
    &  $\coker(J_n)$   & \\
\hline
95 & 2, 3 & 2, 653, 2351, 4513, 56039, 10610063, 13264529, 31184907679, \\
    & 	 &  59862819377, 140737488355327, 153298748932447906241 \\ 
\hline
96 & 2 & \\ 
\hline
97 & 2 & 2 \\ 
\hline
98 & 2 & \\ 
\hline
99 & 2, 3 & 2, 127, 417202699, 4432676798593, \\
	& 	& 	562949953421311, 47464429777438199 \\ 
\hline
100 & 2, 3 & \\ 
\hline
\end{tabular}
\caption{The prime factors of $\Theta_n^{bp}$, $90 \leq n \leq 100$.}
\end{figure}

\bibliographystyle{alpha}
\bibliography{master}

\end{document}